\documentclass{amsart}
\usepackage{amssymb}
\usepackage{amsfonts}
\usepackage{hyperref}

\newtheorem{theorem}{Theorem}
\theoremstyle{plain}

\newtheorem{conjecture}{Conjecture}
\newtheorem{corollary}{Corollary}

\newtheorem{lemma}{Lemma}

\newtheorem{proposition}{Proposition}
\newtheorem{remark}{Remark}

\numberwithin{equation}{section}

\begin{document}
\title[An affirmative answer to an open problem]{An affirmative answer to an open problem on Ramanujan's asymptotic formula of zero-balanced hypergeometric function}
\author{Miao-Kun Wang}
\address{Miao-Kun Wang, Department of mathematics, Huzhou University,
Huzhou, Zhejiang, China, 313000}
\email{wmk000@126.com}
\urladdr{https://orcid.org/0000-0002-0895-7128}
\author{Zhen-Hang Yang$^{\ast \ast }$}
\address{Zhen-Hang Yang, Department of Science and Technology, State Grid Zhejiang Electric Power Company Research Institute, Hangzhou, Zhejiang, China, 310014}
\email{yzhhang1963@outlook.com; yzhkm@163.com}
\urladdr{https://orcid.org/0000-0002-2719-4728}
\author{Tie-Hong Zhao}
\address{Tie-Hong Zhao, School of mathematics, Hangzhou Normal University,
Hangzhou, Zhejiang, China, 311121}
\email{tiehong.zhao@hznu.edu.cn}
\urladdr{https://orcid.org/0000-0002-6394-1049}
\date{}

\thanks{This research was supported by the Natural Science Foundation of Zhejiang Province (LY24A010011).}

\thanks{$^{\ast \ast }$Corresponding author}

\begin{abstract}
In this paper, by using recurrence method and new properties of beta
function and Ramanujan $R$-function, we prove 
the absolute monotonicity
result for a function related to Ramanujan asymptotic formula of zero-balanced hypergeometric function, which gives
an affirmative answer to an open problem proposed by Anderson Vamanamurthy and Vuorinen \cite{Anderson-CIIQM-JW-1997} in 1997.
\end{abstract}

\maketitle

\section{Introduction}
For real numbers $a,b,c$ with $-c\neq 0,-1,-2,\cdots$, the Gaussian hypergeometric function (series) \cite{AS1964} is defined by
\begin{equation}
F(a,b;c;x):={_2F_{1}}(a,b;c;x)=\sum_{n=0}^\infty\frac{(a)_n(b)_n}{(c)_n}\frac{x^n}{n!}\quad \text{for}\quad |x|<1,
\label{hygs}
\end{equation}
where $(a)_0=1$ for $a\neq0$ and
\begin{equation*}
(a)_n=a(a+1)(a+2)\cdots(a+n-1)=\frac{\Gamma(a+n)}{\Gamma(a)}
\end{equation*}for $n\in\mathbb{N}$. Here $\Gamma(x)$ $(x>0)$ is the classical gamma function. In the case $c=a+b$,  $F(a,b;c;x)$ is called zero-balanced. 

The series \eqref{hygs} was firstly introduced by Euler two centuries ago, and later was extensively studied by Gauss since 1812. In the nineteen century, many other famous mathematicians, including  Jacobi, Kummer, Fuchs, Riemann, Legendre, Klein, Schwarz et al., had made significant contributions from different points of view. Specifically, in 1900-1920s, the Indian mathematical genius Ramanujan studied deeply the Gaussian hypergeometric function, and then applied it to number theory and combinatorics, a lot of elegant identities had been written down in his unpublished notebook, however, most of them are lack of the complete proofs (cf. \cite{ASKEY}). In 1985-1998, Berndt edited five notebooks \cite{Berndt1,Berndt2,Berndt3,Berndt4,Berndt5}, in which the  reconstructed proofs of many identities are given in succession, and this also made Ramanujan's results widely accessible. Until now, the Gaussian hypergeometric function has been applied in many fields of mathematics and physics, such as differential equation \cite{Whittaker,Yoshida}, quantum theory \cite{Eckart-1930,Mace-Hellberg-1995}, number theory \cite{Borweins,BBG1995,Shen-RJ-2013}, topology and geometry \cite{Beukers-1989,Beukers-2007,Deligne-Mostow} and geometric function theory \cite{AQVV-Pacific-2000,AVV2001,AVV2007,QV2005}. In particular, since the 1980s, the Gaussian hypergeometric function has been occurred frequently in the theory of conformal and quasiconformal maps, many conformal invariants and distortion theorems of quasiconformal maps all depend on this function. For this, the interesting reader is referred to  Anderson-Vamanamurthy-Vuorinen book "\textit{Conformal Invariants, Inequalities, and Quasiconformal Maps}" \cite{Anderson-CIIQM-JW-1997}, and the recent published papers \cite{AR,Anderson-Sugawa-Va-Vu-2009,Anderson-Sugawa-Va-Vu-2010,Bhayo-Vuorinen,Ma-Qiu-Jiang,QV1999,WCS-2016,Zhang}.

As is well known, the zero-balanced hypergeometric function has a logarithmic singularity at $x=1$, and it satisfies the famous Ramanujan's asymptotic formula
\begin{equation}
F\left(a,b;a+b;x\right)=\frac{R(a,b)-\ln(1-x)}{B(a,b)} +O\left((1-x)\ln(1-x)\right) \text{ \ as }x\rightarrow 1^{-}
	\label{F-Raf}
\end{equation}%
(cf. \cite{Berndt2}),
where%
\begin{equation*}
	B\equiv B\left( a,b\right) =\frac{\Gamma \left( a\right) \Gamma \left(
		b\right) }{\Gamma \left( a+b\right) }\text{ \ and \ }R\equiv R\left(
	a,b\right) =-\psi \left( a\right) -\psi \left( b\right) -2\gamma 
\end{equation*}%
are the beta function and Ramanujan $R$-function (or Ramanujan constant), respectively, here $\psi\left( x\right) =d[\ln\Gamma \left(x\right)]/dx$ is the psi function
and $\gamma =-\psi \left( 1\right) $ is the Euler constant.

In 1997, Anderson, Vamanamurthy and Vuorinen proposed nine open problems in 
\cite[Appendix G]{Anderson-CIIQM-JW-1997}, where the second one is related to \eqref{F-Raf} and states that

\textbf{Open Problem AVV}. \emph{Let} $a,b\in (0,1)$\emph{\ with }$a+b<1$%
\emph{\ and define}
\begin{equation*}
G(a,b,x)=\frac{B\left(a,b\right) F\left( a,b;a+b;x\right) }{\ln \left[
d/\left(1-x\right)\right]},\quad d=e^{R\left(a,b\right)}.
\end{equation*}
\emph{Is it true that the function}%
\begin{equation}
Q(a,b;x)\equiv \frac{G(a,b,x)-1}{1-x}  \label{Q}
\end{equation}%
\emph{has a Maclaurin expansion }$\sum_{n=0}^{\infty }\alpha _{n}x^{n}$\emph{%
\ with non-negative coefficients }$\alpha _{n}$\emph{? A positive answer
would refine \cite[\emph{Theorem 1.52 (1)}]{Anderson-CIIQM-JW-1997}.}

The same problem also appeared in \cite[Open problem 2.3]{Qiu-NMJ-154-1999}, and \cite[3.7 Open problem]{Vuorinen 1}, \cite[2.7 Open problem]{Vuorinen 2}. 

To our knowledge, the origin of Open problem AVV can be traced back to 1992. That year,
Anderson, Vamanamurthy and Vuorinen \cite[Conjecture 3.1 (3)]%
{Anderson-SIAM-JMA-23-1992} made the following conjecture.

\begin{conjecture}
\label{C-AVV}Let $\mathcal{K}(r)=(\pi/2)F(1/2,1/2;1;r^2)$ be the complete elliptic integral of the
first kind. The function
$r\mapsto\left[\mathcal{K}(r)/\ln(4/r')-1\right]/{r'}^2$
is increasing from $(0,1)$ onto $\left( \pi /\ln 16-1,1/4\right)$, where $r'=\sqrt{1-r^2}$ for $r\in(0,1)$.
\end{conjecture}

Instead of proving the monotonicity of the function in Conjecture 1, a few years later, mathematical researchers showed that the double inequality
\begin{equation}
1+\left( \frac{\pi }{\ln 16}-1\right) r^{\prime 2}<\frac{\mathcal{K}\left(
r\right) }{\ln \left( 4/r^{\prime }\right) }<1+\frac{1}{4}r^{\prime 2},
\label{K><}
\end{equation}%
or equivalently,%
\begin{equation*}
\frac{\pi }{\ln 16}-1<Q\left( \frac{1}{2},\frac{1}{2};r^{2}\right) =\frac{%
\mathcal{K}\left( r\right) /\ln \left( 4/r^{\prime }\right) -1}{r^{\prime 2}}%
<\frac{1}{4},
\end{equation*}%
holds for all $r\in \left(0,1\right)$, where the right-hand side of inequality (\ref{K><}) was proved
by Qiu and Vamanamurthy in \cite{Qiu-SIAM-JMA-27(3)-1996}, while the left-hand
one is due to Alzer in \cite{Alzer-MPCPS-124(2)-1998}. Also, Alzer proved that
both the upper and lower bounds in (\ref{K><}) are sharp. In 2015, by means of some combinations
of monotonic functions and L'Hospital Monotone Rule (cf. \cite[Theorem 1.25]{Anderson-CIIQM-JW-1997}, or Lemma \ref{L-LMR} in Section 3),
Wang, Chu and Qiu \cite{Wang-JMAA-429-2015} extended (\ref{K><}) to the
case of generalized elliptic integral. Precisely, they proved the following
theorem.

\noindent \textbf{Theorem A.} \emph{For} $a\in(0,1/2]$, \emph{let} $R(a)=R(a,1-a)$ and 
$\mathcal{K}_{a}\left(r\right)
=(\pi/2)F( a,1-a;1;r^{2})$ \emph{be the
generalized elliptic integral of the first kind  for } $r\in(0,1)$. \emph{Then the double
inequality}%
\begin{equation}
\alpha <Q(a,1-a;r^{2})=\frac{\mathcal{K}_{a}\left( r\right) }{r^{\prime
2}\sin \left( \pi a\right) \ln \left[ e^{R\left(a\right)/2}/r^{\prime
}\right]}-\frac{1}{r^{\prime 2}}<\beta \label{I-WCQ}
\end{equation}
\emph{holds for all }$a\in (0,1/2]$\emph{\ and }$r\in \left( 0,1\right) $%
\emph{\ if and only if }$\alpha \leq \alpha _{0}=\pi /\left[ R\left(
a\right) \sin \left( \pi a\right) \right] -1$\emph{\ and }$\beta \geq \beta
_{0}=a\left( 1-a\right) $\emph{.}

In 2017, by using L'Hospital Monotone Rule and the
unimodal type monotone rule for the ratio of two power series proved in \cite%
{Yang-JMAA-428-2015}, Yang and Chu \cite{Yang-MIA-20-2017} verified the
monotonicity of the function $r\mapsto Q(a,1-a;r^{2})$ on $\left( 0,1\right) 
$, that is, the following theorem.

\medskip
\noindent \textbf{Theorem B. }The function%
\begin{equation}
r\mapsto Q(a,1-a;r^{2})=\frac{\mathcal{K}_{a}\left( r\right) }{r^{\prime
2}\sin \left( \pi a\right) \ln \left( e^{R\left( a\right) /2}/r^{\prime
}\right) }-\frac{1}{{r'}^2}  \label{Y(r)}
\end{equation}%
is strictly increasing from $\left( 0,1\right) $ onto $\left( \alpha
_{0},\beta _{0}\right) $, where $\alpha _{0}$ and $\beta _{0}$ are defined
in Theorem A.

\begin{remark}
In particular, when $a=1/2$ in (\ref{Y(r)}), the function%
\begin{equation*}
r\mapsto Q\left( \frac{1}{2},\frac{1}{2};r^{2}\right) =\frac{\mathcal{K}%
\left( r\right) /\ln \left( 4/r^{\prime }\right) -1}{r^{\prime 2}}
\end{equation*}%
is strictly increasing from $\left( 0,1\right) $ onto $\left( \pi /\ln
16-1,1/4\right) $, which solved Conjecture \ref{C-AVV} in \cite[Conjecture
3.1(3)]{Anderson-SIAM-JMA-23-1992}.
\end{remark}

Later, in 2019, following the method in  \cite{Yang-MIA-20-2017},  Wang, Chu and Zhang \cite{WCZ-2019} proved the monotonicity property of $Q(a,b;x)$ on $(0,1)$ for arbitrary $(a,b)\in\{a>0,b>0\}$, and derived the following Theorem C as a corollary.


\medskip
\noindent \textbf{Theorem C.} For $a,b>0$ with $a+b\leq 1$, the function $Q(a,b;x)$ defined in \eqref{Q} is strictly increasing from $\left(0,1\right) $ onto $\left(\alpha
_{0}^*,\beta _{0}^*\right) $, where $\alpha_{0}^*=\alpha_{0}^*(a,b)=[B(a,b)-R(a,b)]/B(a,b)$ and $\beta_{0}^*=\beta_{0}^*(a,b)=ab$.

\begin{remark}
From the published literature, Theorem C seems to be the best result related
to Open problem AVV. However, it is far from this problem to be solved. In
other words, Open problem AVV is still open.
\end{remark}

Recall that a function $f$ is called absolutely monotonic on an interval $I$
if  $f^{\left( n\right) }\left( x\right) \geq 0$ 
for $x\in I$ and every $n\geq 0$ (cf. \cite{Widder-LT-1946}). Obviously, if $f(x)$ is a power series converging on $(0,r)$ $(r>0)$, then $f(x)$ is
absolutely monotonic on $(0,r)$ if and only if all coefficients of $f(x)$
are nonnegative. A function $f$ is said to be completely monotonic on an interval $I$ if $\left( -1\right) ^{n}f^{\left( n\right) }\left( x\right) \geq 0$
for $x\in I$ and every $n\geq 0$ (cf. \cite{Bernstein-AM-52-1928,	Widder-LT-1946}).

Recently, several absolutely monotonic functions related to asymptotic
formula for the complete elliptic integral of the first kind were discovered
in succession, see for example, \cite{Yang-JMI-15-2021}, \cite%
{Tian-RM-77-2022}, \cite{Yang-AMS-42-2022}, \cite{Zhao-arXiv-2024}. These
motivate us to seek to solve this problem. For this, we firstly introduce a sequence $\{w_{n}\}$, and define a function $Q_{p}(x)$. For $n\in\mathbb{N}\cup \{0\}$, let
\begin{equation}
w_{n}=\frac{\left( a\right) _{n}\left( b\right) _{n}}{\left( a+b\right)_{n}n!}=
\frac{1}{B}\frac{\Gamma\left( n+a\right) \Gamma\left(n+b\right) }{%
\Gamma\left(n+a+b\right)n!},  \label{wn}
\end{equation}
then 
\begin{equation}
	w_{n+1}=\frac{\left( n+a\right) \left( n+b\right) }{\left( n+1\right) \left(
		n+a+b\right) }w_{n}, \label{wn-rr}
\end{equation}
and use of \eqref{wn} gives
\begin{equation*}
	F\left( a,b;a+b;x\right) =\sum_{n=0}^{\infty }\frac{\left( a\right)
		_{n}\left( b\right) _{n}}{\left( a+b\right)_{n}}\frac{x^{n}}{n!}%
	=\sum_{n=0}^{\infty }w_{n}x^{n}.
\end{equation*}
For $p>0$, let%
\begin{equation}
	Q_{p}(x)\equiv Q_{p}(a,b;x)=\frac{1}{1-x}\left[ \frac{B\left( a,b\right)
		F\left( a,b;a+b;x\right) }{p-\ln \left( 1-x\right) }-1\right]
	=\sum_{n=0}^{\infty }\alpha _{n}x^{n}.  \label{Qp}
\end{equation}

In what follows, we shall state our
strategy of solving the \textbf{Open Problem AVV}, which can be divided into several steps.

\textbf{Step 1}: Establishing the recurrence relation of coefficients $%
\alpha _{n}$. To this end, we let
\begin{equation}
U^{\ast }\left( x\right) =B\left( a,b\right) F\left( a,b;a+b;x\right) -p+\ln
\left( 1-x\right) =\sum_{n=0}^{\infty }u_{n}^{\ast }x^{n},  \label{U*}
\end{equation}
\begin{equation}
U\left( x\right)=\frac{U^{\ast }\left( x\right) }{1-x}
=\sum_{n=0}^{\infty }u_{n}x^{n},\label{U}
\end{equation}
\begin{equation}
V\left( x\right)=p-\ln \left( 1-x\right) =\sum_{n=0}^{\infty }v_{n}x^{n},\label{V}
\end{equation}
where
\begin{align}
u_{n}=&\sum_{k=0}^{n}u_{k}^{\ast }\text{, \ \ }u_{0}^{\ast }=B-p\text{, \ }%
u_{k}^{\ast }=Bw_{k}-\frac{1}{k}\text{ for }k\geq 1,  \label{un} \\
v_{n}=&\frac{1}{n}\text{ for }n\geq 1\text{ and }v_{0}=p.  \label{vn}
\end{align}%
Then from \eqref{Qp}--\eqref{V}  $Q_{p}(x)$ can be written as%
\begin{equation*}
Q_{p}(x)=\frac{U^{\ast }\left( x\right) /\left( 1-x\right) }{V\left(
x\right) }=\frac{U\left( x\right) }{V\left( x\right) }=\sum_{n=0}^{\infty
}\alpha _{n}x^{n}.
\end{equation*}
Using Cauchy product formula we obtain that $\alpha _{0}=B/p-1$ and, for $n\geq 1$,
\begin{equation}
v_{0}\alpha _{n}=u_{n}-\sum_{k=1}^{n}v_{k}\alpha _{n-k}=u_{n}-\sum_{k=1}^{n}%
\frac{\alpha _{n-k}}{k},  \label{an-rr}
\end{equation}%
which, together with \eqref{un} and \eqref{vn}, indicates that%
\begin{equation}
\alpha _{1}=-\frac{\mathcal{S}\left( p\right) }{p^{2}}\text{, \ where \ }%
\mathcal{S}\left( p\right) =p^{2}-B\frac{a+b+ab}{a+b}p+B\text{.}  \label{S}
\end{equation}%
Replacing $n$ by $n+1$ in (\ref{an-rr}) yields%
\begin{equation}
p\alpha _{n+1}=u_{n+1}-\sum_{k=0}^{n}\frac{\alpha _{n-k}}{k+1}.
\label{an+1-rr}
\end{equation}%
Eliminating $\alpha _{0}$ from (\ref{an-rr}) and (\ref{an+1-rr}) reveals that%
\begin{equation*}
p\alpha _{n+1}-\frac{n}{n+1}p\alpha _{n}=u_{n+1}-\frac{n}{n+1}u_{n}-\alpha
_{n}+\sum_{k=1}^{n-1}\left( \frac{n}{n+1}-\frac{k}{k+1}\right) \frac{\alpha_{n-k}}{k},
\end{equation*}%
which, by an arrangement, gives%
\begin{equation}
p\left( n+1\right) \alpha _{n+1}=\left[ \left( n+1\right) u_{n+1}-nu_{n}%
\right] +n\left( p-\frac{n+1}{n}\right) \alpha _{n}+\sum_{k=1}^{n}\frac{n-k}{%
k\left(k+1\right) }\alpha _{n-k}.  \label{an+1-rra}
\end{equation}

{\bf Step 2}: Under the assumption of $(a,b)\in(0,1)$ with $a+b\leq 1$, determine the range of $p$ such that 
$\alpha_{n}\geq 0$ for all $n\geq 0$. Due to \eqref{an+1-rra}, it is easy for us to find that, for any  $p\in E_{1}\cap E_{2}\cap E_{3}\cap E_{4}$, where%
\begin{align}
	E_{1}=&\left\{ p:p>0,\alpha _{0}=\frac{B}{p}-1\geq 0\right\}=\left\{p:0<p<B\right\} ,  \label{E1} \\
	E_{2}=&\left\{ p:p>0,\alpha _{1}=-\dfrac{\mathcal{S}\left( p\right) }{p^{2}%
	}\geq 0\right\} =\left\{ p:p>0,\mathcal{S}\left( p\right) \leq 0\right\} ,
	\label{E2} \\
	E_{3}=&\left\{ p:p>0,\left( n+1\right) u_{n+1}-nu_{n}\geq 0\text{ for all }%
	n\geq 1\right\} ,  \label{E3} \\
	E_{4}=&\left\{ p:p>0,p-\dfrac{n+1}{n}\geq 0\text{ for all }n\geq 1\right\},  \label{E4}
\end{align}%
$\alpha _{n}\geq 0$ for
all $n\geq 0$ by mathematical induction. In other words, if $p\in \bigcap_{i=1}^{4}E_{i}$, then $Q_{p}(x)
$ is absolutely monotonic on $\left( 0,1\right)$. We shall show that $\bigcap_{i=1}^{4}E_{i}=\{p: 2\leq p\leq R(a,b)\}$ according to three substeps.

\textbf{Substep 2.1}: proving that $E_{3}\cap E_{4}=\left\{ p:2\leq p\leq R\left(
a,b\right) \right\}$.

\textbf{Substep 2.2}: proving that $E_{3}\cap E_{4}\subset E_{1}$.

\textbf{Substep 2.3}: proving that $E_{3}\cap E_{4}\subset E_{2}$.

\medskip
From Steps 1 and 2, we deduce our main result immediately.

\begin{theorem}
\label{MT}Let $p>0$ and $a,b\in (0,1)$ with $a+b\leq 1$. If $2\leq p\leq
R\left( a,b\right) $, then the function $x\mapsto Q_{p}(x)$ defined by (\ref%
{Qp}) is absolutely monotonic on $\left( 0,1\right) $. In other words, $%
Q_{p}(x)$ has a Maclaurin expansion $\sum_{n=0}^{\infty }\alpha _{n}x^{n}$
with non-negative coefficients $\alpha _{n}$.
\end{theorem}

\begin{remark}
Taking $p=R\left( a,b\right) $ in Theorem \ref{MT} gives a positive answer
to \textbf{Open problem AVV}.
\end{remark}

\medskip
\section{Proof of Theorem \protect\ref{MT}}

In this section, we prove Theorem \ref{MT} step by step. It should be noted,
that some properties involving $R\left( a,b\right) $ and $B\left( a,b\right) 
$, which are applied to prove Substep 2.3, are postponed to the next section.

\subsection{Substep 2.1: $E_{3}\cap E_{4}=\left\{ p:2\leq p\leq R\left(
a,b\right) \right\} $}

\begin{proposition}[Substep 2.1]
\label{P-S2}Let $u_{n}$, $E_{3}$ and $E_{4}$ be defined by (\ref{un}), (\ref%
{E3}) and (\ref{E4}), respectively. If $a,b\in (0,1)$ with $a+b\leq 1$, then 
\begin{equation*}
E_{3}\cap E_{4}=\left\{ p:2\leq p\leq R\left( a,b\right) \right\} .
\end{equation*}
\end{proposition}

\begin{proof}
It is clear that%
\begin{equation*}
E_{4}=\left\{ p:p>0,p-\dfrac{n+1}{n}\geq 0\text{ for all }n\geq 1\right\}
=\left\{ p:p\geq 2\right\} .
\end{equation*}%
We next prove that $E_{3}=\left\{ p:0<p\leq R\left( a,b\right) \right\}$. For $n\in\mathbb{N}$, let $d_{n}=(n+1)u_{n+1}-nu_{n}$. Note that, for each $p\in E_{3}$, $d_{n}\geq 0$ for $n\geq 1$ and thereby  $\lim_{n\rightarrow \infty }d_{n}\geq 0$. Thus, if we can show that 
\begin{equation*}
	\lim_{n\rightarrow \infty }d_{n}=R\left(a,b\right)-p.
\end{equation*}
Then we can conclude that $p\leq R(a,b)$ for any $p\in E_{3}$. Actually, by (\ref{un}),%
\begin{equation}
d_{n}=n\left( u_{n+1}-u_{n}\right) +u_{n+1}=nu_{n+1}^{\ast
}+\sum_{k=0}^{n+1}u_{k}^{\ast }.  \label{dn}
\end{equation}%
On the one hand, it is seen from (\ref{un}), (\ref{wn}) and the asymptotic
formula \cite[p. 257, Eq. (6.1.46)]{Abramowitz-HMFFGMT-1970} 
\begin{equation}
\frac{\Gamma \left( n+x\right) }{\Gamma \left( n+y\right) }\thicksim n^{x-y}%
\text{ \ as }n\rightarrow \infty ,  \label{gr-af}
\end{equation}%
that%
\begin{equation*}
nu_{n+1}^{\ast }=nBw_{n+1}-\frac{n}{n+1}=n\frac{\Gamma \left( n+1+a\right)
\Gamma \left( n+1+b\right) }{\Gamma \left( n+2\right) \Gamma \left(
n+1+a+b\right) }-\frac{n}{n+1}\rightarrow 0
\end{equation*}%
as $n\rightarrow \infty $. On the other hand, by (\ref{U*}) and the
asymptotic formula (\ref{F-Raf}), we have%
\begin{align*}
\lim_{n\rightarrow \infty }\sum_{k=0}^{n+1}u_{k}^{\ast }
=&\sum_{k=0}^{\infty }u_{k}^{\ast }=\lim_{x\rightarrow 1^{-}}U^{\ast
}\left( x\right)  \\
=&\lim_{r\rightarrow 1^{-}}\left[ B\left( a,b\right) F\left(
a,b;a+b;x\right)+\ln\left(1-x\right)-p\right]\\
=&R\left( a,b\right) -p.
\end{align*}

Conversely, if $p\leq R\left( a,b\right) $, then we can deduce that $p\in
E_{3}$, or equivalently, $d_{n}=\left( n+1\right) u_{n+1}-nu_{n}\geq 0$ for
all $n\geq 1$. In fact, using (\ref{dn}) and the recurrence formula (\ref%
{wn-rr}), we have%
\begin{align*}
d_{n}-d_{n-1}=&nu_{n+1}^{\ast }+\sum_{k=0}^{n+1}u_{k}^{\ast }-\left(
n-1\right) u_{n}^{\ast }-\sum_{k=0}^{n}u_{k}^{\ast }=\left( n+1\right)
u_{n+1}^{\ast }-\left( n-1\right) u_{n}^{\ast } \\
=&\left( n+1\right) \left[ B\frac{\left( n+a\right) \left( n+b\right) }{%
\left( n+1\right) \left( n+a+b\right) }w_{n}-\frac{1}{n+1}\right]-\left(
n-1\right) \left( Bw_{n}-\frac{1}{n}\right)  \\
=&\frac{1}{n}\left( nBw_{n}\frac{n+a+b+ab}{n+a+b}-1\right) :=\frac{1}{n}%
\left( \vartheta _{n}-1\right) ,
\end{align*}%
\begin{align*}
\frac{\vartheta _{n+1}}{\vartheta {n}}-1=&\frac{\left(n+1\right) B\frac{%
n+1+a+b+ab}{n+1+a+b}}{nB\frac{n+a+b+ab}{n+a+b}}\frac{\left( n+a\right)
\left( n+b\right) }{\left( n+1\right) \left( n+a+b\right) }-1 \\
=&\frac{ab\left( a+1\right) \left( b+1\right) }{n\left( a+b+n+ab\right)
\left( a+b+n+1\right) }>0.
\end{align*}%
From this it thus can be seen that the sequence $\left\{ \vartheta
_{n}\right\} _{n\geq 1}$ is increasing, and hence, by the asymptotic formula
(\ref{gr-af}), 
\begin{equation*}
\vartheta _{n}<\lim_{n\rightarrow \infty }\vartheta _{n}=\lim_{n\rightarrow
\infty }\left( nBw_{n}\frac{n+a+b+ab}{n+a+b}\right) =1.
\end{equation*}%
This implies that $d_{n}-d_{n-1}<0$ for $n\geq 2$, that is, $\left\{
d_{n}\right\} _{n\geq 1}$ is decreasing. It then follows that 
\begin{equation*}
d_{n}>\lim_{n\rightarrow \infty }d_{n}=R\left( a,b\right) -p\geq 0
\end{equation*}%
for all $n\geq 1$.

Finally, we also have to prove that $E_{3}\cap E_{4}\neq \emptyset $ if $%
a,b\in (0,1)$ with $a+b\leq 1$, which suffices to prove that $R\left(
a,b\right) >2$. In fact, since $\psi ^{\prime }\left( x\right) >0$ and $\psi
^{\prime \prime }\left( x\right) <0$ for $x>0$, it follows from the Jensen
inequality that 
\begin{align*}
R\left(a,b\right)=&-\psi \left( a\right) -\psi \left( b\right) -2\gamma
\geq -2\psi \left( \frac{a+b}{2}\right) -2\gamma  \\
\geq &-2\psi \left( \frac{1}{2}\right) -2\gamma =4\ln 2>2,
\end{align*}%
which completes the proof.
\end{proof}

\subsection{Substep 2.2: $E_{3}\cap E_{4}\subset E_{1}$}

\begin{proposition}[Substep 2.2]
Let $E_{1}$, $E_{3}$ and $E_{4}$ be defined by (\ref{E1}), (\ref{E3}) and (%
\ref{E4}), respectively. If $a,b\in (0,1)$ with $a+b\leq 1$, then $E_{3}\cap
E_{4}\subset E_{1}$.
\end{proposition}
\begin{proof}
	It has been shown in Proposition \ref{P-S2} that $E_{3}\cap E_{4}=\left[
	2,R\left( a,b\right) \right] $. Since $E_{1}=(0,B\left( a,b\right) ]$, to
	prove that $E_{3}\cap E_{4}\subset E_{1}$, it suffices to prove that $B\left( a,b\right) >R\left(a,b\right)$. It was proved in  \cite[Theoerm 1.52(2)]{Anderson-CIIQM-JW-1997} that $B(a,b)>R(a,b)$ for all $a,b>0$, therefore the proof is completed.
\end{proof}

\subsection{Substep 2.3: $E_{3}\cap E_{4}\subseteq E_{2}$}

\begin{lemma}
\label{L-S(2)<0}Let $\mathcal{S}\left( p\right) $ be defined by (\ref{S}).
If $a,b>0$ and $a+b\leq 1$, then $\mathcal{S}\left( 2\right) <0$.
\end{lemma}

\begin{proof}
By (\ref{S}), we have
\begin{equation}
\mathcal{S}\left(2\right) =4-\frac{a+b+2ab}{a+b}B.  \label{S(2)}
\end{equation}%
It was prove in \cite[Theorem 1.1]{Zhao-arXiv-2023} that%
\begin{equation*}
B\left( x,y\right) >\frac{x+y}{xy}\left( 1-\frac{2xy}{x+y+1}\right) 
\end{equation*}%
for $x,y\in \left( 0,1\right) $. Applying the above inequality to \eqref{S(2)}
gives%
\begin{align*}
\mathcal{S}\left( 2\right)  
<&4-\frac{a+b+2ab}{a+b}\frac{a+b}{ab}\left( 1-%
\frac{2ab}{a+b+1}\right)  \\
=&-\frac{\left( a+b\right) ^{2}+a+b-2ab-4ab\left( a+b\right)-4a^2b^2}{ab\left( a+b+1\right) } \\
:=&-\frac{S\left(a,b\right) }{ab\left(a+b+1\right) }.
\end{align*}%
By the elementary inequality $ab\leq \left( a+b\right) ^{2}/4$ for $a,b>0$,
we have%
\begin{align*}
S\left(a,b\right)&\geq \left( a+b\right) ^{2}+a+b-\frac{1}{2}\left(
a+b\right) ^{2}-\left( a+b\right) ^{3}-\frac{1}{4}\left( a+b\right) ^{4} \\
&=\frac{1}{4}c\left( 4+2c-4c^{2}-c^{3}\right) >0
\end{align*}%
for $c=a+b\in (0,1]$. Then $\mathcal{S}\left( 2\right) <0$ if $a,b>0$ and $%
a+b\leq 1$, which completes the proof.
\end{proof}

\begin{lemma}
\label{L-S(R)<0}Let $\mathcal{S}\left(p\right) $ be defined by (\ref{S}).
If $a,b>0$ and $a+b\leq 1$, then $\mathcal{S}\left( R\right)<0$.
\end{lemma}

\begin{proof}
Without loss of generality, we assume that $0<a\leq b$. Let $b=c-a$, then $c\leq 1$ and $a\in
(0,c/2]$. The inequality $\mathcal{S}\left( R\right)<0$ can be written
as%
\begin{equation*}
\mathcal{S}\left( R\right) =R(a,c-a)^{2}-R(a,c-a)B(a,c-a)\frac{c+a(c-a)}{c}%
+B(a,c-a)<0,
\end{equation*}%
or equivalently,%
\begin{equation}
\mathcal{S}_{c}\left( a\right) :=\frac{R(a,c-a)}{B(a,c-a)}\frac{%
B(a,c-a)-R(a,c-a)}{a(c-a)}+\frac{a(c-a)R(a,c-a)-c}{ca(c-a)}>0  \label{Sc}
\end{equation}%
for $a\in (0,c/2]$ and any fixed $c\in (0,1]$.

As shown in Proposition \ref{P-qc-pd}, Proposition \ref{P-B-R-hm} and
Corollary \ref{C-R1w-d} in the following section, for any fixed $c\in(0,1]$, the function

$(i)$ $x\mapsto R(x,c-x)/B(x,c-x)$ is positive and strictly	decreasing on $(0,c/2]$;
	
$(ii)$ $x\mapsto \left[ B(x,c-x)-R(x,c-x)\right] /\left[ x\left( c-x\right)	\right]$ is positive and strictly decreasing on $(0,c/2]$;
	
$(iii)$ $x\mapsto R(x,c-x) -c/\left[ x\left(c-x\right)\right]$
is strictly decreasing on $(0,c/2]$.

\noindent Therefore, the function
\begin{equation*}
a\mapsto \mathcal{S}_{c}\left( a\right) =\frac{R(a,c-a)}{B(a,c-a)}\frac{%
	B(a,c-a)-R(a,c-a)}{a(c-a)}+\frac{1}{c}\left[R(a,c-a)-\frac{c}{a(c-a)}\right]
\end{equation*}%
is strictly decreasing on $(0,c/2]$. This yields%
\begin{align*}
\mathcal{S}_{c}\left( a\right)>&\mathcal{S}_{c}\left( \frac{c}{2}\right) =%
\frac{R\left( c/2,c/2\right)}{B\left( c/2,c/2\right) }\frac{B\left(
c/2,c/2\right) -R\left( c/2,c/2\right) }{\left( c/2\right) ^{2}} \\
&+\frac{1}{2}\left[ \frac{1}{(c/2)}R\left( \frac{c}{2},\frac{c}{2}\right) -%
\frac{2}{\left(c/2\right)^{2}}\right] :=\mathcal{S}^{\ast }\left( \frac{c}{%
2}\right) .
\end{align*}%
Thus, if we can prove that%
\begin{equation*}
\mathcal{S}^{\ast }\left( x\right) =q\left( x\right) d\left( x\right)
+r\left( x\right) >0
\end{equation*}%
for $x\in \left(0,1/2\right]$, then the inequality $\mathcal{S}\left(
R\right)<0$ follows, where%
\begin{equation*}
q\left( x\right) =\frac{R(x,x)}{B(x,x)}\text{, \ }d\left( x\right) =\frac{%
B\left( x,x\right) -R\left( x,x\right) }{x^{2}}\text{ \ and \ }r\left(
x\right)=\frac{1}{2}\left[\frac{xR\left( x,x\right) -2}{x^{2}}\right].
\end{equation*}%
By Proposition \ref{P-B,R-m}, $x\mapsto q\left( x\right) ,d\left( x\right) $
are strictly decreasing and positive on $\left(0,1/2\right]$, and while $x\mapsto r\left(
x\right) $ is strictly increasing $\left(0,1/2\right]$. We now
distinguish three cases to prove that $\mathcal{S}^{\ast }\left( x\right) >0$
for $x\in (0,1/2]$.

\textbf{Case 1}: $x\in (0,1/4]$. Then 
\begin{equation*}
q\left( x\right) \geq q\left( \frac{1}{4}\right)=0.984382\cdots,\quad
d\left( x\right) \geq d\left( \frac{1}{4}\right)=1.853168\cdots,
\end{equation*}
\begin{equation*}
r\left( x\right) \geq r\left( 0^{+}\right)=-\frac{\pi^{2}}{6},
\end{equation*}
which implies that%
\begin{align*}
\mathcal{S}^{\ast}\left( x\right)\geq&q\left( \frac{1}{4}\right)
d\left(\frac{1}{4}\right) +r\left( 0^{+}\right)  \\
=&0.984382\cdots\times 1.853168\cdots-\frac{\pi^{2}}{6}=0.179\cdots>0.
\end{align*}

\textbf{Case 2}: $x\in (1/4,7/16]$. Then
\begin{equation*}
q\left( x\right) \geq q\left( \frac{7}{16}\right) =0.920714\cdots,\quad
d\left( x\right) \geq d\left(\frac{7}{16}\right) =1.558534\cdots, 
\end{equation*}
\begin{equation*}
r\left( x\right)\geq r\left( \frac{1}{4}\right)=2\pi +12\ln{2}-16,
\end{equation*}
which implies that
\begin{align*}
\mathcal{S}^{\ast}\left( x\right) \geq &q\left(\frac{7}{16}\right)
d\left( \frac{7}{16}\right) +r\left( \frac{1}{4}\right) \\
=&0.920714\cdots\times 1.558534\cdots+\left(2\pi +12\ln{2}-16\right)
=0.035\cdots>0.
\end{align*}

\textbf{Case 3}: $x\in (7/16,1/2]$. Then
\begin{equation*}
q\left( x\right) \geq q\left( \frac{1}{2}\right)=\frac{4\ln{2}}{\pi}, \quad
d\left( x\right)\geq d\left(\frac{1}{2}\right)=4\pi-16\ln 2>0,	
\end{equation*}
\begin{equation*}
r\left( x\right) \geq r\left( \frac{7}{16}\right) =-1.265376\cdots,	
\end{equation*}
which implies that
\begin{align*}
\mathcal{S}^{\ast }\left( x\right)  \geq &q\left( \frac{1}{2}\right)
d\left( \frac{1}{2}\right) +r\left( \frac{7}{16}\right)  \\
=&\frac{4 \ln{2}}{\pi }\times \left( 4\pi-16\ln 2\right)
-1.265376\cdots=0.112\cdots>0.
\end{align*}%
This completes the proof.
\end{proof}

\begin{proposition}[Substep 2.3]
\label{P-S4}Let $E_{2}$, $E_{3}$ and $E_{4}$ be defined by (\ref{E2}), (\ref%
{E3}) and (\ref{E4}), respectively. If $a,b\in (0,1)$ with $a+b\leq 1$, then 
$E_{3}\cap E_{4}\subseteq E_{2}$.
\end{proposition}

\begin{proof}
Since%
\begin{equation*}
\mathcal{S}\left( p\right)=p^{2}-B\frac{a+b+ab}{a+b}p+B
\end{equation*}%
is a quadratic polynomial of $p$,  $\mathcal{S}(0)=B(a,b)>0$, and  $\mathcal{S}(2)<0$, $\mathcal{S}(R)<0$ by Lemmas \ref{L-S(2)<0} and \ref{L-S(R)<0},
$\mathcal{S}\left( p\right) $ has two
zeros $p_{1}=p_{1}\left( a,b\right)$ and $p_{2}=p_{2}\left(a,b\right)$ which satisfies $0<p_{1}<2<R(a,b)<p_{2}$. Then
\begin{equation*}
E_{2}=\left\{ p:p>0,\mathcal{S}\left( p\right) \leq 0\right\} =\left[
p_{1},p_{2}\right],
\end{equation*}%
and $\left[ 2,R\right] \subseteq \left[ p_{1},p_{2}\right]$, thereby completes the proof.
\end{proof}

\subsection{Proof of Theorem \protect\ref{MT}}

\begin{proof}
By Propositions \ref{P-S2}--\ref{P-S4} we have that%
\begin{equation*}
E_{1}\cap E_{2}\cap E_{3}\cap E_{4}=E_{3}\cap E_{4}=\left\{ p:2\leq p\leq
R\left( a,b\right) \right\} .
\end{equation*}%
This implies that $\alpha _{0}>0$, $\alpha _{1}>0$, $\left( n+1\right)
u_{n+1}-nu_{n}\geq 0$ for all $n\geq 1$, $[p-( n+1)/n]\geq 0$ for all $n\geq 1$. Assume that $\alpha _{n}>0$ for $1\leq
n\leq m$. Then $\alpha _{m+1}>0$ via (\ref{an+1-rra}). By induction, we
arrive at $\alpha _{n}>0$ for all $n\geq 1$. This together with $\alpha
_{0}>0$ proves the absolute monotonicity of $Q_{p}\left( x\right) $ on $%
\left( 0,1\right) $ if $2\leq p\leq R\left( a,b\right) $. The proof of
Theorem \ref{MT} is done.
\end{proof}

\medskip
\section{Some properties of Beta function and Ramanujan $R$-function}

In order to verify the positivity of $\mathcal{S}_{c}\left( a\right) $
defined by (\ref{Sc}), we need some properties of $R\left( a,b\right) $ and $%
B\left( a,b\right) $ in the case of $\left( a,b\right) =\left( x,c-x\right) $
for fixed $c>0$ and any $x\in \left( 0,c\right) $. In this case, we denote by%
\begin{align}
B_{c}\left( x\right) & =B\left( x,c-x\right) =\frac{\Gamma \left( x\right)
\Gamma \left( c-x\right) }{\Gamma \left( c\right) },  \label{Bc} \\
R_{c}\left( x\right) & =R\left( x,c-x\right) =-\psi \left( x\right) -\psi
\left( c-x\right) -2\gamma .  \label{Rc}
\end{align}
Before proving properties of $B_{c}\left( x\right) $ and $R_{c}\left(
x\right) $, let us recall several basic knowledge points. The first point is
the integral representations of beta function (cf. \cite[p. 258, Eq. (6.2.1)]%
{Abramowitz-HMFFGMT-1970}):%
\begin{equation}
B(p,q)=\int_{0}^{1}t^{p-1}(1-t)^{q-1}dt=\int_{0}^{\infty }\frac{t^{p-1}}{%
(1+t)^{p+q}}dt=\int_{0}^{1}\frac{t^{p-1}+t^{q-1}}{(1+t)^{p+q}}dt,
\label{B-ir}
\end{equation}%
where the last representation follows from
\begin{align*}
\int_{1}^{\infty }\frac{t^{p-1}}{(1+t)^{p+q}}dt\overset{t=1/s}{%
=\!=\!=\!}\int_{0}^{1}\frac{s^{q-1}}{(1+s)^{p+q}}ds.
\end{align*}

The second point is the series, integral representations and recurrence formulas of psi and
polygamma functions:%
\begin{equation}
\psi (x)+\gamma =\int\limits_{0}^{\infty }\frac{e^{-t}-e^{-xt}}{1-e^{-t}}%
dt=\int_{0}^{1}\frac{1-t^{x-1}}{1-t}dt,  \label{psi-ir}
\end{equation}%
\begin{equation}
\psi ^{(n)}(x)=\left\{ 
\begin{array}{ll}
-\gamma -\dfrac{1}{x}+\sum\limits_{k=0}^{\infty }\dfrac{x}{k(k+x)} & \text{%
for }n=0, \\ 
(-1)^{n+1}n!\sum\limits_{k=0}^{\infty }\dfrac{1}{(x+k)^{n+1}} & \text{for }%
n\in \mathbb{N}
\end{array}%
\right.   \label{psi-sr}
\end{equation}%
and
\begin{equation}
\psi ^{(n)}(x+1)-\psi ^{(n)}(x)=\left( -1\right) ^{n}\frac{n!}{x^{n+1}} \ \text{for }%
n\in \mathbb{N}\cup\{0\} \label{rfpsi}
\end{equation}
(cf. \cite[p.259-260, Eq. (6.3.16), Eq. (6.3.22), Eq. (6.4.6), Eq. (6.4.10)]{Abramowitz-HMFFGMT-1970}).

The third point is the L'Hospital monotonic rule \cite%
{Vamanamurthy-JMAA-183-1994}, that is, the following lemma.

\begin{lemma}
\label{L-LMR}Let $-\infty <a<b<\infty $, and let $f,g:[a,b]\rightarrow 
\mathbb{R}$ be continuous functions that are differentiable on $\left(
a,b\right) $, with $f\left( a\right) =g\left( a\right) =0$ or $f\left(
b\right) =g\left( b\right) =0$. Assume that $g^{\prime }(x)\neq 0$ for each $%
x$ in $(a,b)$. If $f^{\prime }/g^{\prime }$ is increasing (decreasing) on $%
(a,b)$ then so is $f/g$.
\end{lemma}

Next we first give the monotonicity of the ratio $q_{c}(x)=R_{c}(x)/B_{c}(x)$.

\begin{proposition}
\label{P-qc-pd}The function $x\mapsto q_{c}(x)=R_{c}(x)/B_{c}(x)$ is
strictly decreasing from $(0,c/2]$ onto $[R_{c}(c/2)/B_{c}(c/2),1)$. In
particular, $q_{c}\left( x\right) $ is positive and strictly decreasing on $%
(0,c/2]$ if and only if $0<c\leq 2$.
\end{proposition}

\begin{proof}
Use of \eqref{Bc}, \eqref{Rc} and (\ref{psi-sr}) yields 
\begin{align*}
\frac{B_{c}^{\prime }(x)}{B_{c}(x)}& =\psi (x)-\psi (c-x)=-\frac{1}{x}+\frac{%
1}{c-x}+\sum_{k=1}^{\infty }\left[ \frac{x}{k(x+k)}-\frac{c-x}{k(c-x+k)}%
\right]  \\
& =\frac{2x-c}{x(c-x)}+\sum_{k=1}^{\infty }\frac{k(2x-c)}{k(c-x+k)(x+k)},
\end{align*}%
\begin{eqnarray*}
R_{c}^{\prime }(x) &=&\psi ^{\prime }(c-x)-\psi ^{\prime
}(x)=\sum_{k=0}^{\infty }\left[ \frac{1}{(c-x+k)^{2}}-\frac{1}{(x+k)^{2}}%
\right]  \\
&=&\sum_{k=0}^{\infty }\frac{(2k+c)(2x-c)}{(c-x+k)^{2}(x+k)^{2}},
\end{eqnarray*}%
\begin{align*}
R_{c}(x)& =\frac{1}{x}-\sum_{k=1}^{\infty }\frac{x}{k(x+k)}+\frac{1}{c-x}%
-\sum_{k=1}^{\infty }\frac{c-x}{k(c-x+k)} \\
& =\frac{c}{x(c-x)}-\sum_{k=1}^{\infty }\frac{ck+2x(c-x)}{k(x+k)(c-x+k)}.
\end{align*}%
Then%
\begin{align}
B_{c}(x)q_{c}^{\prime }(x)& =R_{c}^{\prime }(x)-R_{c}(x)\frac{B_{c}^{\prime
}(x)}{B_{c}(x)}=\psi ^{\prime }(c-x)-\psi ^{\prime }(x)-R_{c}(x)\left[ \psi
(x)-\psi (c-x)\right]   \notag \\
& =\sum_{k=0}^{\infty }\frac{(2k+c)(2x-c)}{(c-x+k)^{2}(x+k)^{2}}-\left[ 
\frac{c}{x(c-x)}-\sum_{k=1}^{\infty }\frac{ck+2x(c-x)}{k(x+k)(c-x+k)}\right] 
\notag \\
& \times \left[ \frac{2x-c}{x(c-x)}+\sum_{k=1}^{\infty }\frac{k(2x-c)}{%
k(c-x+k)(x+k)}\right]   \notag
\end{align}%
\begin{align*}
& =\sum_{k=0}^{\infty }\frac{(2k+c)(2x-c)}{(c-x+k)^{2}(x+k)^{2}}-\frac{c}{%
x(c-x)}\frac{2x-c}{x(c-x)} \\
& -\frac{c}{x(c-x)}\sum_{k=1}^{\infty }\frac{k(2x-c)}{k(c-x+k)(x+k)}+\frac{%
2x-c}{x(c-x)}\sum_{k=1}^{\infty }\frac{ck+2x(c-x)}{k(x+k)(c-x+k)} \\
& +\sum_{k=1}^{\infty }\frac{ck+2x(c-x)}{k(x+k)(c-x+k)}\sum_{k=1}^{\infty }%
\frac{k(2x-c)}{k(c-x+k)(x+k)}
\end{align*}%
\begin{align*}
& =\left( 2x-c\right) \sum_{k=1}^{\infty }\frac{2k+c}{(c-x+k)^{2}(x+k)^{2}}%
+2\left( 2x-c\right) \sum_{k=1}^{\infty }\frac{1}{k(x+k)(c-x+k)} \\
& +\left( 2x-c\right) \sum_{k=1}^{\infty }\frac{ck+2x(c-x)}{k(x+k)(c-x+k)}%
\sum_{k=1}^{\infty }\frac{k}{k(c-x+k)(x+k)}.
\end{align*}%
This means that sgn$q_{c}^{\prime }(x)=$sgn$\left( 2x-c\right) $, and hence,
the function $q$ is decreasing on $(0,c/2]$. An easy computation leads to $%
\lim_{x\rightarrow 0^{+}}q_{c}(x)=1$ and%
\begin{equation*}
q_{c}\left( \frac{c}{2}\right) =2\frac{\Gamma \left( c\right) }{\Gamma
\left( c/2\right) ^{2}}\left[ \psi \left( 1\right) -\psi \left( \frac{c}{2}%
\right) \right] .
\end{equation*}%
Obviously, $q_{c}\left( c/2\right) \geq 0$ if and only if $0<c\leq 2$. This
completes the proof.
\end{proof}

Second, we present the (higher order) monotonicity of the difference $\delta
_{c}(x)=B_{c}(x)-R_{c}(x)$ on $\left( 0,c\right) $.

\begin{proposition}
\label{P-B-R-hm}Let $c>0$. Define the function $\delta _{c}$ on $(0,c)$ by%
\begin{equation}
\delta _{c}(x)=B_{c}(x)-R_{c}(x).  \label{dc}
\end{equation}%
Then the following statements hold:

$(i)$ $\delta _{c}(x)$ has the integral representation: 
\begin{equation*}
\delta _{c}(x)=-2\int_{0}^{1}\frac{t^{c/2}[(1+t)^{c}-1+t]\cosh \left[ \left(
x-c/2\right) \ln t\right] -t(1+t)^{c}}{t\left( 1-t\right) (1+t)^{c}}dt
\end{equation*}%
with $\delta _{c}(0^{+})=0$.

$(ii)$ For $n\in \mathbb{N}$, we have%
\begin{equation*}
\delta _{c}^{\left( 2n-1\right) }(x)=\left\{ 
\begin{array}{cc}
>0 & \text{for }x\in \left( 0,c/2\right) , \\ 
<0 & \text{for }x\in \left( c/2,c\right) 
\end{array}%
\right. \text{ \ and \ }\delta _{c}^{\left( 2n\right) }(x)<0\text{ for }x\in
\left( 0,c\right) .
\end{equation*}%
Consequently, $\delta _{c}(x)>0$ for $x\in \left( 0,c\right) $, $x\mapsto
\delta _{c}^{\prime }\left( x\right) $ is completely monotonic on $\left(
0,c/2\right) $, and $x\mapsto -\delta _{c}^{\prime }\left( x\right) $ is
absolutely monotonic on $\left( c/2,c\right) $.

$(iii)$ The function $x\mapsto \delta _{c}(x)/\left[ x\left( c-x\right) \right]
$ is positive and strictly decreasing (resp. increasing) on $(0,c/2]$ (resp. 
$[c/2,c)$).
\end{proposition}

\begin{proof}
(i) Using the third integral representation of (\ref{B-ir}) and the second one of (%
\ref{psi-ir}) we obtain 
\begin{align}
\delta _{c}(x)& =B_{c}(x)-R_{c}(x)=-\int_{0}^{1}\frac{\left[ \left(
t+1\right) ^{c}-1+t\right] (t^{x}+t^{c-x})-2t(1+t)^{c}}{t(1-t)(1+t)^{c}}dt
\label{IROD} \\
& =-2\int_{0}^{1}\frac{t^{c/2}\left[ \left( t+1\right) ^{c}-1+t\right] \cosh %
\left[ \left( x-c/2\right) \ln t\right] -t(1+t)^{c}}{t(1-t)(1+t)^{c}}dt. 
\notag
\end{align}
Since
\begin{equation*}
\lim_{x\rightarrow 0^{+}}B_{c}(x)-\frac{1}{x}
=\lim_{x\rightarrow 0^{+}}R_{c}(x)-\frac{1}{x}=-\gamma-\psi\left(c\right) ,
\end{equation*}
we have that%
\begin{equation*}
\delta _{c}(0^{+})=\lim_{x\rightarrow 0^{+}}\left[ \left( B_{c}(x)-\frac{1}{x%
}\right) -\left( R_{c}(x)-\frac{1}{x}\right) \right] =0.
\end{equation*}

(ii) According to \eqref{IROD}, we have that, for $n\in \mathbb{N}$,
\begin{align}
\delta _{c}^{\left(2n-1\right)}(x)=& -2\int_{0}^{1}\frac{t^{c/2}\left[\left(t+1\right)^{c}-1+t\right]\left(\ln t\right)^{2n-1}\sinh\left[\left( x-c/2\right)\ln{t}\right] }{t(1-t)(1+t)^{c}}dt \label{dc-2n-1'}\\
&\left\{ 
\begin{array}{cc}
	>0 &\text{for }x\in \left( 0,c/2\right) , \\ 
	<0 &\text{for }x\in \left( c/2,c\right),
\end{array}
\right.\notag
\end{align}
\begin{equation}
\delta _{c}^{\left(2n\right) }(x)=-2\int_{0}^{1}\frac{t^{c/2}\left[ \left(
t+1\right) ^{c}-1+t\right] \left( \ln t\right) ^{2n}\cosh \left[ \left(
x-c/2\right) \ln t\right] }{t(1-t)(1+t)^{c}}dt<0  \label{dc-2n'}
\end{equation}%
for $x\in \left( 0,c\right) $. Since $\delta _{c}^{\prime }(x)>0$ for $x\in
\left( 0,c/2\right) $ and $\delta _{c}^{\prime }(x)<0$ for $x\in \left(
c/2,c\right) $, we derive that%
\begin{equation}
0=\delta (0^{+})<\delta _{c}(x)<\delta _{c}(c/2)=B_{c}(c/2)-R_{c}(c/2)
\label{dc>0}
\end{equation}%
for $x\in \left( 0,c/2\right) $. Due to $\delta _{c}(c-x)=\delta _{c}(x)$,
we see that the inequality (\ref{dc>0}) also holds for $x\in \left(
c/2,c\right) $.

(iii) Let $p\left( x\right) =x\left( c-x\right) $. Then $\delta _{c}\left(
0^{+}\right) =p\left( 0^{+}\right) =\delta _{c}\left( c^{-}\right) =p\left(
c^{-}\right) =0$. By (\ref{dc-2n-1'}) and $p^{\prime }\left( x\right) =c-2x$%
, we see that $\delta _{c}^{\prime }\left( c/2\right) =p^{\prime }\left(
c/2\right) =0$. Then by (\ref{dc-2n'}) and (\ref{dc-2n-1'}),%
\begin{equation*}
\frac{\delta _{c}^{\prime \prime }\left( x\right) }{p^{\prime \prime }\left(
x\right) }=-\frac{1}{2}\delta _{c}^{\prime \prime }\left( x\right) \text{ \
and \ \ }\left[ \frac{\delta _{c}^{\prime \prime }\left( x\right) }{%
p^{\prime \prime }\left( x\right) }\right] ^{\prime }=-\frac{1}{2}\delta
_{c}^{\prime \prime \prime }\left( x\right) \left\{ 
\begin{array}{cc}
<0 & \text{for }x\in \left( 0,c/2\right) , \\ 
>0 & \text{for }x\in \left( c/2,c\right) .%
\end{array}%
\right. 
\end{equation*}%
Applying L'Hospital monotonic rule twice, one can easily see that $\delta _{c}\left( x\right)
/p\left( x\right) $ is decreasing on $\left( 0,c/2\right) $ and increasing
on $\left( c/2,c\right) $. The positivity of $\delta _{c}\left( x\right)
/p\left( x\right) $ on $\left( 0,c\right) $ follows from (\ref{dc>0}),
thereby completing the proof.
\end{proof}

Third, we consider the higher order monotonicity of the function 
\begin{equation*}
\tilde{R}_{n}\left( x\right) =R_{c}\left( x\right) -\sum_{k=0}^{n-1}\frac{%
2k+c}{\left( k+x\right) \left( k+c-x\right) }
\end{equation*}%
on $\left( 0,c\right) $, which is contained in the following proposition.

\begin{proposition}
\label{P-Rnw-hm} For $n\in \mathbb{N}$ and $m\in \mathbb{N}\cup\{0\}$, Then $\tilde{R}_{n}\left( x\right) $ satisfies%
\begin{equation*}
\tilde{R}_{n}^{\left( 2m+1\right) }\left( x\right) \left\{ 
\begin{array}{ll}
<0 & \text{for }x\in \left( 0,c/2\right) , \\ 
>0 & \text{for }x\in (c/2,c]%
\end{array}%
\right. \text{ \ and \ }\tilde{R}_{n}^{\left( 2m+2\right) }\left( x\right) >0%
\text{ for }x\in (0,c].
\end{equation*}%
In particular, $\tilde{R}_{n}\left( x\right) $ is decreasing on $(0,c/2)$,
increasing on $(c/2,c]$, and $\tilde{R}_{n}\left( x\right) <0$ for all $x\in
(0,c]$.
\end{proposition}

\begin{proof} Due to \eqref{rfpsi}, it is not difficult to verify that
\begin{equation*}
\tilde{R}_{n}\left( x\right) =-\psi \left( n+c-x\right) -\psi \left(
n+x\right) -2\gamma.
\end{equation*}
Using \eqref{psi-ir} and then differentiating give
\begin{equation*}
\tilde{R}_{n}^{\prime }\left( x\right)
=\int\limits_{0}^{\infty }\frac{e^{-\left( n+c-x\right) t}-e^{-\left(
n+x\right) t}}{1-e^{-t}}tdt=2\int\limits_{0}^{\infty }\frac{t\sinh \left[
\left( x-c/2\right) t\right] }{\left( 1-e^{-t}\right) e^{\left( n+c/2\right)
t}}dt.
\end{equation*}%
Then for $m\in \mathbb{N}\cup\{0\}$,%
\begin{align*}
\tilde{R}_{n}^{\left( 2m+1\right) }\left( x\right) =&
2\int\limits_{0}^{\infty }\frac{t^{2m+1}\sinh \left[ \left( x-c/2\right) t%
\right] }{\left( 1-e^{-t}\right) e^{\left( n+c/2\right) }}dt\left\{ 
\begin{array}{cc}
<0 & \text{if }x\in \left( 0,c/2\right) , \\ 
>0 & \text{if }x\in (c/2,c],%
\end{array}%
\right.  \\
\tilde{R}_{n}^{\left( 2m+2\right) }\left( x\right) =&
2\int\limits_{0}^{\infty }\frac{t^{2m+2}\cosh \left[ \left( x-c/2\right) t%
\right] }{\left( 1-e^{-t}\right) e^{\left( n+c/2\right) }}dt>0,\text{\ for }%
x\in (0,c].
\end{align*}%
Moreover, for $n\in\mathbb{N}$,
\begin{equation*}
\lim_{x\rightarrow 0^{+}}\tilde{R}_{n}(x)=\lim_{x\rightarrow c^{-}}\tilde{R}%
_{n}(x)=-\psi \left( n+c\right) -\psi \left( n\right) -2\gamma \leq -2\psi
(1)-2\gamma =0.
\end{equation*}%
This completes the proof.
\end{proof}

Taking $n=1$ in Proposition \ref{P-Rnw-hm} gives the following corollary.

\begin{corollary}
\label{C-R1w-d}The function%
\begin{equation*}
x\mapsto \tilde{R}_{1}\left( x\right) =R_{c}\left( x\right) -\frac{c}{%
x\left( c-x\right) }
\end{equation*}%
is strictly decreasing on $(0,c/2)$, and strictly increasing on $(c/2,c]$.
Moreover, 
\begin{equation*}
\tilde{R}_{1}\left( 0^{+}\right) =\tilde{R}_{1}\left( c^{-}\right) =-\psi
(1+c)-\gamma ,\quad \tilde{R}_{1}\left( \frac{c}{2}\right) =-2\psi \left( 
\frac{c}{2}\right) -2\gamma -\frac{4}{c}.
\end{equation*}
\end{corollary}

Fourth, the following three monotonicity properties are needed to prove Lemma \ref{L-S(R)<0}, which was proved in \cite[Theorems 4--6]{Yang-2024-submitted}.

\begin{proposition} \label{P-B,R-m}
$(i)$ The the function $x\mapsto R(x,x)/B(x,x)$ is strictly
decreasing and concave from $\left( 0,\infty \right) $ onto $\left( -\infty
,1\right) $.

$(ii)$ The function $x\mapsto \left[ B\left( x,x\right) -R\left( x,x\right) %
\right] /x^{2}$ is decreasing from $\left( 0,\infty \right) $ onto $\left(
0,2\zeta \left( 3\right) \right) $.

$(iii)$ The function $x\mapsto \left[ xR\left( x,x\right) -2\right] /x^{2}$ is
increasing from $\left( 0,\infty \right) $ onto $\left( -\pi ^{2}/3,0\right) 
$.
\end{proposition}

\medskip
\section{Concluding remarks}

In this paper, we proved a stronger result than the positive answer to Open problem AVV by using
recurrence method and new properties of $B\left( a,b\right) $ and $R\left(
a,b\right) $, which states that, if $a,b>0$ with $a+b\leq 1$ and $2\leq
p\leq R\left( a,b\right) $, then the function%
\begin{equation*}
Q_{p}(x)\equiv Q_{p}(a,b;x)=\frac{1}{1-x}\left[ \frac{B\left( a,b\right)
F\left( a,b;a+b;x\right) }{p-\ln \left( 1-x\right) }-1\right] 
\end{equation*}%
is absolutely monotonic on $\left( 0,1\right) $. This solved an open problem
that had not been solved for nearly 30 years.

Finally, we present several remarks.

\begin{remark}
Using the increasing property and convexity of $Q_{R}(x)$ on $(0,1)$ with
\begin{equation*}
Q_{R}(0)=\frac{B\left( a,b\right) }{R\left( a,b\right) }-1\text{ \ and \ }%
Q_{R}(1^{-})=ab,
\end{equation*}%
we have%
\begin{align*}
1+\left[ \frac{B\left( a,b\right) }{R\left( a,b\right)}-1\right] (1-x)<\frac{B\left(a,b\right) F\left( a,b;a+b;x\right)}{R\left(
a,b\right)-\ln\left( 1-x\right) }\\
<1+ab(1-x)-\left[ab-\left(\frac{B(a,b)}{R(a,b)}-1\right)\right](1-x)^2
\end{align*}%
for $x\in\left(0,1\right)$, which extend and improve the double
inequality \eqref{I-WCQ}. Furthermore, it is apparent from \eqref{wn}, \eqref{un} and \eqref{an+1-rra} that $\alpha_{n}$ can be calculated. Our main result implies that, for any $n\in\mathbb{N}$, the function
\begin{equation*}
x\mapsto\frac{Q_{R}(x)-\sum\limits_{k=0}^{n}\alpha_{k}x^k}{x^{n+1}}
\end{equation*}
is strictly increasing from $(0,1)$ onto $(\alpha_{n+1},ab-\sum_{k=0}^{n}\alpha_{k})$. Consequently, for $x\in(0,1)$,
\begin{align*}
1+(1-x)\left(\alpha_{n+1}x^{n+1}+\sum_{k=0}^{n}\alpha_{k}x^k\right)<\frac{B\left(a,b\right) F\left( a,b;a+b;x\right)}{R\left(
	a,b\right)-\ln\left( 1-x\right)}\\
<1+(1-x)\left[\left(ab-\sum_{k=0}^{n}\alpha_{k}\right)x^{n+1}+\sum_{k=0}^{n}\alpha_{k}x^{k}\right].
\end{align*}
Particularly, letting $n=2$, $a=1/2$, $b=1/2$ and next replacing $x$ by $r^2$, we obtain
\begin{align*}
&1+(1-r^2)\left(\alpha_{3}^*r^{6}+\sum_{k=0}^{2}\alpha_{k}^*r^{2k}\right)< \frac{\mathcal{K}(r)}{\ln(4/r')}\\
&<1+(1-r^2)\left[\left(\frac{1}{4}-\sum_{k=0}^{2}\alpha_{k}^*\right)r^{6}
+\sum_{k=0}^{2}\alpha_{k}^*r^{2k}\right].
\end{align*}
where
\begin{align*}
	\alpha_{0}^*=&\frac{\pi}{4 \ln{2}}-1, \quad \alpha_{1}^*=\frac{-16 \left(\ln{2}\right)^{2}+5 \pi  \left(\ln{2}\right)-\pi}{16\left(\ln{2}\right)^{2}},\\
	\alpha_{2}^*=&\frac{-256 \left(\ln{2}\right)^{3}+89 \pi  \left(\ln{2}\right)^{2}-28 \pi  \left(\ln{2}\right)+4 \pi}{256\left(\ln{2}\right)^{3}},\\
	\alpha_{3}^*=&\frac{-3072 \left(\ln{2}\right)^{4}+1143 \pi  \left(\ln{2}\right)^{3}-451 \pi  \left(\ln{2}\right)^{2}+108 \pi  \left(\ln{2}\right)-12 \pi}{3072\left(\ln{2}\right)^{4}}.
\end{align*}

\end{remark}

\begin{remark}
Taking $p=R\left( a,b\right) $ and $b=1-a$ in Theorem \ref{MT}, we see that
the function defined by (\ref{Y(r)}) is absolutely monotonic on $\left(
0,1\right) $.
\end{remark}

\begin{remark} Anderson, Vamanamurthy and Vuorinen \cite{Anderson-SIAM-JMA-23-1992} also conjectured that the inequality	
\begin{equation}\label{Kanother}
\mathcal{K}(\sqrt{x})<\ln\left(1+\frac{4}{\sqrt{1-x}}\right)
-\left(\ln{5}-\frac{\pi}{2}\right)(1-\sqrt{x})
\end{equation}
holds for all $x\in(0,1)$. The conjecture was proved in \cite{Qiu-SIAM-JMA-27(3)-1996} by Qiu, Vamanamurthy and Vuorinen. Recently, Yang, Tian \cite{YT-AADM-2019} and Wang, Chu, Li et al. \cite{WCL-AADM-2020} proved that $x\mapsto \mathcal{K}(\sqrt{x})/\ln(1+4/\sqrt{1-x})$ is strictly increasing and convex on $(0,1)$ with the range $(\pi/(2\ln{5}),1)$, and $x\mapsto \frac{d^2}{dx^2}\left[\mathcal{K}(\sqrt{x})-\ln\left(1+4/\sqrt{1-x}\right)\right]$ is absolutely monotonic on $(0,1)$. As applications of these results, several sharp inequalities for the complete elliptic integral of the first kind have been derived, which improve \eqref{Kanother}.
It is not difficult to find that $\ln(1+4/\sqrt{1-x})$ is another good approximation of $\mathcal{K}(\sqrt{x})$ similar to $\ln(4/\sqrt{1-x})$. Computer simulation and experiments show that
\begin{equation*}
x\mapsto \frac{1-\mathcal{K}(\sqrt{x})/\ln(1+4/\sqrt{1-x})}{1-x}
\end{equation*}
is absolutely monotonic on $(0,1)$. If this conclusion is proved to be true, it will also yields better estimates for $\mathcal{K}(\sqrt{x})$ than \eqref{Kanother}. 
Furthermore, it is natural to ask the following question, for what values of positive numbers $a$ and $b$, the function 
\begin{equation*}
x\mapsto \displaystyle\frac{1-BF(a,b;a+b;x)/\ln[1+e^{R}/(1-x)]}{1-x}
\end{equation*}
is absolutely monotonic on $(0,1)$. 
\end{remark}

\begin{remark}
Those properties of $B_{c}(x)$ and $R_{c}(x)$ proved in Section 3 are of independent interest. For example, by Proposition \ref{P-B-R-hm} and the series representation of $R_{c}(x)$ proved in \cite[Theorem 3.2]{BWC-2023}, we also obtain the series expansions of $B_{c}(x)$, which will extend the corresponding result of $B(x,1-x)$. Certainly, many other analytical properties and functional inequalities of $B_{c}(x)$, $R_{c}(x)$ and their combinations are worth exploring.
\end{remark}


\begin{thebibliography}{99}
\bibitem{Anderson-CIIQM-JW-1997} G. D. Anderson, M. K. Vamanamurthy, M. K.
Vuorinen, \textit{Conformal Invariants, Inequalities, and Quasiconformal Maps}, John Wiley \& Sons, New
York, 1997.

\bibitem{AS1964}
M. Abramowitz, I. S. Stegun,  \textit{Handbook of Mathematical Functions with Formulas, Graphs, and Mathematical Tables}, U.S. Government Printing Office, Washington, 1964.


\bibitem{ASKEY} R. Askey, \textit{Ramanujan and hypergeometric and basic hypergeometric series}, Russ. Math. Surv., \textbf{45}(1990), no.1, 37--86.

\bibitem{Berndt1} B. C. Berndt, \emph{Ramanujan's Notebooks, Part I}, Springer-Verlag, New York, 1985.

\bibitem{Berndt2}B. C. Berndt, \emph{Ramanujan's Notebooks, Part II}, Springer-Verlag, New York, 1989.

\bibitem{Berndt3}B. C. Berndt, \emph{Ramanujan's Notebooks, Part III}, Springer-Verlag, New York, 1991.

\bibitem{Berndt4} B. C. Berndt, \emph{Ramanujan's Notebooks, Part IV}, Springer-Verlag, New York, 1994.

\bibitem{Berndt5} B. C. Berndt, \emph{Ramanujan's Notebooks, Part V}, Springer-Verlag, New York, 1998.

\bibitem{Whittaker}E. T. Whittaker, G. N. Watson, \textit{A Course of Modern Analysis, 4th ed.}, Cambridge Univ. Press, London, 1962.

\bibitem{Yoshida} M. Yoshida. \textit{Fuchsian differential equations. with special emphasis on the Gauss-Schwarz theory}, Aspects of Mathematics, E11. Friedr. Vieweg \& Sohn, Braunschweig, 1987.

\bibitem{Eckart-1930}
C. Eckart, \textit{The penetration of a potential barrier by electrons}, Phys. Rev. \textbf{35}(1930), no.11, 1303--1309. https://doi.org/10.1103/PhysRev.35.1303

\bibitem{Mace-Hellberg-1995} R. L. Mace, M. A. Hellberg,\textit{ A dispersion function for plasmas containing superthermal particles}, Physics of Plasmas \textbf{2}(1995), no.6, 2098--2109. DOI:10.1063/1.871296

\bibitem{Borweins} J. M. Borwein, P. B. Borwein, \textit{Pi and the AGM}, John Wiley \& Sons, New York, 1987.

\bibitem{BBG1995}
B. C. Berndt, S. Bhargava, F. G. Garvan,
\textit{Ramanujan's theories of elliptic functions to alternative bases},
Trans. Amer. Math. Soc. {\bf 347}(1995), no. 11, 4163--4244. 
https://doi.org/10.2307/2155035

\bibitem{Shen-RJ-2013} 
L. C. Shen. \textit{A note on Ramanujan's identities
involving the hypergeometric function $_{2}F_{1}(1/6,5/6;1;z)$},
Ramanujan J., \textbf{30}(2013), 
no. 2, 211--222. https://doi.org/10.1007/s11139-011-9360-8

\bibitem{Beukers-1989} 
F. Beukers, G. Heckman, \textit{Monodromy for the hypergeometric function $_nF_{n-1}$}. 
Invent. Math. \textbf{95}(1989), 
no. 2, 325--354. https://doi.org/10.1007/BF01393900

\bibitem{Beukers-2007} 
F. Beukers, 
\textit{Gauss's hypergeometric function}(Holzapfel, Rolf-Peter et al. ed.),
in: Arithmetic and Geometry around Hypergeometric Functions. Lecture notes of a CIMPA summer school held at Galatasaray University, Istanbul, Turkey, June 13--25, 2005, Progress in Mathematics, 260, 23--42(2007).

\bibitem{Deligne-Mostow} 
P. Deligne, G. D. Mostow, 
\textit{Monodromy of hypergeometric functions and non-lattice integral monodromy}, 
Publ. Math., Inst. Hautes Etud. Sci. \textbf{63}(1986), 5--89. https://doi.org/10.1007/BF02831622

\bibitem{AQVV-Pacific-2000} G. D. Anderson, S. L. Qiu, M. K. Vamanamurthy, M. Vuorinen,
\textit{Generalized elliptic integrals and modular equations}, 
Pacific J. Math. \textbf{192}(2000),
no. 1, 1-37.
https://doi.org/10.2140/pjm.2000.192.1

\bibitem{AVV2001} 
G. D. Anderson, M. K. Vamanamurthy, M. Vuorinen, 
\textit{Topics in special functions}, in: Papers on Analysis: A volume dedicated to Olli Martio on the occasion of his 60th birthday, Report Univ. Jyv\"{a}skyl\"{a} \textbf{83}(2001), 5--26.

\bibitem{AVV2007}
G. D. Anderson, M. K. Vamanamurthy, M. Vuorinen, 
\textit{Topics in special functions II}, Conform. Geom. Dyn. \textbf{11}(2007), 250--270.
https://doi.org/10.1090/S1088-4173-07-00168-3

\bibitem{QV2005} 
S. L. Qiu, M. Vuorinen,
\textit{Chapter 14 Special functions in geometric function theory}, 
In: Handbook of Complex Analysis: Geometric Function
Theory, Vol.2, Elsevier Sci. B. V., Amsterdam, 2005, 621--659.


\bibitem{AR} 
H. Alzer and K. Richards, 
\textit{On the modulus of the Gr\"{o}tzsch ring}, 
J. Math. Anal. Appl. \textbf{432}(2015), no. 1, 134--141.
https://doi.org/10.1016/j.jmaa.2015.06.057

\bibitem{Anderson-Sugawa-Va-Vu-2009} 
G. D. Anderson, T. Sugawa, M. K. Vamanamurthy, M. Vuorinen, \textit{Hypergeometric functions and hyperbolic metric}. 
Comput. Methods Funct. Theory \textbf{9}(2009), 
no. 1, 269--284.
https://doi.org/10.1007/BF03321727

\bibitem{Anderson-Sugawa-Va-Vu-2010} 
G. D. Anderson, T. Sugawa, M. K. Vamanamurthy, M. Vuorinen, \textit{Twice-punctured hyperbolic sphere with a conical singularity and generalized elliptic integral}, 
Math. Z. \textbf{266}(2010), 
no. 1, 181--191.
https://doi.org/10.1007/s00209-009-0560-5

\bibitem{Bhayo-Vuorinen} 
B. A. Bhayo, M. Vuorinen, 
\textit{On generalized complete elliptic integrals and modular functions},
Proc. Edinb. Math. Soc., II. Ser. \textbf{55}(2012), 
no. 3, 591--611. https://doi.org/10.1017/S0013091511000356

\bibitem{Ma-Qiu-Jiang}
X. Y. Ma, S. L. Qiu, H. B. Jiang, \textit{Monotonicity theorems and inequalities for the H\"{u}bner function with applications},
J. Math. Anal. Appl. \textbf{498}(2021), no. 2, Article ID 124977, 22 pages. https://doi.org/10.1016/j.jmaa.2021.124977


\bibitem{QV1999} S. L. Qiu, M. Vuorinen, \textit{Infinite products and normalized quotients of hypergeometric functions}, 
SIAM J. Math. Anal. \textbf{30}(1999), no. 5, 1057--1075. https://doi.org/10.1137/S0036141097326805

\bibitem{WCS-2016} 
M. K. Wang, Y. M. Chu, Y. Q. Song, \textit{Asymptotical formulas for Gaussian and generalized hypergeometric functions},
Appl. Math. Comput. \textbf{276}(2016), 44--60.
https://doi.org/10.1016/j.amc.2015.11.088

\bibitem{Zhang}X. H. Zhang, 
\textit{On the generalized modulus}, 
Ramanujan J. \textbf{43}(2017), 
no.2, 405--413.
https://doi.org/10.1007/s11139-015-9746-0


\bibitem{Qiu-NMJ-154-1999} 
S. L. Qiu, M. Vuorinen, 
\textit{Landen inequalities
for hypergeometric functions}, \emph{Nagoya Math. J.} \textbf{154} (1999),
31--56.
https://doi.org/10.1017/S0027763000025290

\bibitem{Vuorinen 1} M. Vuorinen, \textit{Geometric properties of quasiconformal maps and special functions},
Bull. Soc. Sci. Lett, \L \'{o}d\'{z} 47, S\'{e}r. Rech. D\'{e}form. \textbf{24}(1997), 7--58.


\bibitem{Vuorinen 2} M. Vuorinen, \textit{Hypergeometric functions in geometric function theory}, in: Special
Functions and Differential equations, Proceedings of a workshop held at The
Institute of Mathematical Sciences, Madras, India January 13-24, 1997, ed. by
K. Srinivasa Rao et al., Allied
Publishers Private Limited, pp. 119-126, 1998.

\bibitem{Anderson-SIAM-JMA-23-1992} 
G. D. Anderson, M. K. Vamanamurthy, M.
Vuorinen, 
\textit{Functional inequalities for hypergeometric functions and complete
elliptic integrals}, 
SIAM J. Math. Anal. \textbf{23}(1992), no. 2, 512--524.
https://doi.org/10.1137/0523025

\bibitem{Qiu-SIAM-JMA-27(3)-1996} 
S. L. Qiu, M. K. Vamanamurthy, 
\textit{Sharp estimates for complete elliptic integrals}, 
SIAM J. Math. Anal.\textbf{27}(1996), no. 3, 823--834.
https://doi.org/10.1137/0527044



\bibitem{Alzer-MPCPS-124(2)-1998} 
H. Alzer, 
\textit{Sharp inequalities for the
complete elliptic integral of the first kind}, 
Math. Proc. Cambridge Philos. Soc. \textbf{124}(1998), 
no. 2, 309--314.
https://doi.org/ 10.1017/S0305004198002692

\bibitem{Wang-JMAA-429-2015} 
M. K. Wang, Y.-M. Chu, S. L. Qiu, 
\textit{Sharp bounds for generalized elliptic integrals of the first kind}, 
J. Math. Anal. Appl. \textbf{429}(2015), no. 2, 744--757.
https://doi.org/10.1016/j.jmaa.2015.04.035

\bibitem{Yang-JMAA-428-2015} 
Z. H. Yang, Y. M. Chu, M. K. Wang,
\textit{Monotonicity criterion for the quotient of power series with applications}, 
J. Math. Anal. Appl. \textbf{428}(2015), no. 1, 587-604.
https://doi.org/10.1016/j.jmaa.2015.03.043

\bibitem{Yang-MIA-20-2017} 
Z. H. Yang, Y. M. Chu, 
\textit{A monotonicity property involving the generalized elliptic integral of the first kind}, 
Math. Inequal. Appl. \textbf{20}(2017), 
no. 3, 729--735.
https://doi.org/10.7153/mia-20-46

\bibitem{WCZ-2019} 
M. K. Wang, Y. M. Chu, W. Zhang, \textit{Monotonicity and inequalities involving
zero-balanced hypergeometric function}, Math. Inequal. Appl. \textbf{22}(2019), 
no. 2, 601--617. https://doi.org/10.7153/mia-2019-22-42

\bibitem{Widder-LT-1946} D. V. Widder. The Laplace Transform. Princeton
University Press, Princeton, 1946.

\bibitem{Bernstein-AM-52-1928} S. N. Bernstein, Sur les fonctions absolument
monotones. Acta Math., \textbf{52}(1928), 1--66. Available online at
https://doi.org/10.1007/BF02592679



\bibitem{Yang-JMI-15-2021} 
Z. H. Yang, J. F. Tian, 
\textit{Absolutely monotonic
functions involving the complete elliptic integrals of the first kind with
applications}, 
J. Math. Inequal. \textbf{15}(2021), 
no. 3, 1299--1310. 
https://doi.org/10.7153/jmi-2021-15-87

\bibitem{Tian-RM-77-2022} 
J. F. Tian, Z. H. Yang, 
\textit{Several absolutely monotonic functions related to the complete elliptic integral of the first kind}, Results Math. \textbf{77}(2022),
no .3,  Article No. 109, 19 pages.
https://doi.org/10.1007/s00025-022-01641-4

\bibitem{Yang-AMS-42-2022} 
Z. H. Yang, J. F. Tian, 
\textit{Absolute monotonicity involving
the complete elliptic integrals of the first kind with applications}, 
Acta Math. Sci. Ser. B, Engl. Ed. \textbf{42}(2022), 
no. 3, 847--864.
https://doi.org/10.1007/s10473-022-0302-x

\bibitem{Zhao-arXiv-2024} 
T. Zhao, Z. H, Yang, 
Absolutely monotonic
functions related to the asymptotic formula for the complete elliptic
integral of the first kind,
\emph{arXiv: 2405.19651} [math.CA].
https://doi.org/10.48550/arXiv.2405.19651

\bibitem{Abramowitz-HMFFGMT-1970} M. Abramowitz, I. A. Stegun (Eds), 
\emph{Handbook of Mathematical Functions with Formulas, Graphs, and
Mathematical Tables}, National Bureau of Standards, Applied Mathematics
Series 55, 9th printing, Washington, 1970.

\bibitem{Zhao-arXiv-2023} T. Zhao, M, Wang, A lower bound for the beta
function, arXiv: 2305.02754 [math.CA].
https://doi.org/10.48550/arXiv.2305.02754


\bibitem{Vamanamurthy-JMAA-183-1994}
M. K. Vamanamurthy, M. Vuorinen,
\textit{Inequalities for means}, 
J. Math. Anal. Appl. \textbf{183}(1994),
no. 1, 155--166. https://doi.org/10.1006/jmaa.1994.1137

\bibitem{Yang-2024-submitted} Z. H. Yang, M. K. Wang, T. H. Zhao,\textit{Some new properties of the beta function and Ramanujan $R$-function}, Ramanujan J. \textbf{67}(2025), Article 12, 25 pages. https://doi.org/10.1007/s11139-025-01062-1

\bibitem{YT-AADM-2019}Z. H. Yang, J. F. Tian, 
\textit{Convexity and monotonicity for elliptic integrals of the first kind and applications},
Appl. Anal. Discrete. Math. \textbf{13}(2019), no. 1, 240--260.
https://doi.org/10.2298/AADM190924020W


\bibitem{WCL-AADM-2020} M. K. Wang, H. H. Chu, Y. M. Li, Y. M. Chu,
\textit{Answers to three conjectures on convexity of three functions
involving complete elliptic integrals of the first kind},
Appl. Anal. Discrete Math. \textbf{14}(2020), 
no. 1, 255-271.
https://doi.org/10.2298/AADM190924020W


\bibitem{BWC-2023} 
Q. Bao, M. K. Wang, Y. M. Chu, \textit{$q$-Ramanujan asymptotic formula and $q$-Ramanujan $R$-function}, 
Acta Math. Sci. \textbf{43A}(2023), no. 6, 1659--1666. (in Chinese)
\end{thebibliography}
\end{document}